\documentclass[12pt,reqno]{amsart}

\usepackage[margin=1in]{geometry}  
\usepackage{amsmath}
\usepackage{amssymb}
\usepackage[pdftex]{graphicx}
\usepackage{verbatim}   
\usepackage{color}    
\usepackage{mathrsfs}  
\usepackage[colorlinks=true]{hyperref}
\hypersetup{urlcolor=blue, citecolor=blue, pdfstartview=FitH}
\hypersetup{pdfstartview={XYZ null null 1.25}}
\usepackage{fontenc}
\usepackage{textcomp}
\usepackage{fix-cm}
\usepackage{array}
\usepackage{enumerate}
\usepackage{latexsym, amsfonts, amssymb}
\usepackage{color}
\usepackage{txfonts}
\usepackage{comment}
\usepackage{subfigure}
\usepackage{accents}

\newcommand{\Diff}{\text{Diff}}
\newcommand{\Lie}{\mathcal{L}}

\numberwithin{equation}{section}

\newtheorem{theorem}{Theorem}[section]
\newtheorem*{maintheorem*}{Main Theorem}
\newtheorem{proposition}[theorem]{Proposition}
\newtheorem{lemma}[theorem]{Lemma}

\DeclareMathOperator{\sgn}{sgn}

\begin{document}

\title{One-parameter solutions of the Euler-Arnold equation on the contactomorphism group}

\author{Stephen C. Preston}
\address{Department of Mathematics, University of Colorado, Boulder, CO 80309-0395} 
\email{Stephen.Preston@colorado.edu}

\author{Alejandro Sarria}
\address{Department of Mathematics, University of Colorado, Boulder, CO 80309-0395} 
\email{Alejandro.Sarria@colorado.edu}

\thanks{The first author is partially supported by NSF grant 1105660.}



\begin{abstract} 
We study solutions of the equation 
$$ g_t-g_{tyy} + 4g^2 - 4gg_{yy} = y gg_{yyy}-yg_yg_{yy}, \qquad y\in\mathbb{R},$$
which arises by considering solutions of the Euler-Arnold equation on a contactomorphism group when the stream function is of the form $f(t,x,y,z) = zg(t,y)$. The equation is analogous to both the Camassa-Holm equation and the Proudman-Johnson equation. We write the equation as an ODE in a Banach space to establish local existence, and we describe conditions leading to global existence and conditions leading to blowup in finite time. 
\end{abstract}

\maketitle

\section{Introduction}

In this brief note we study regularity of solutions to the Cauchy problem 
\begin{equation}
\label{main}
\begin{cases}
g_t-g_{yyt}+4g^2-4gg_{yy}=ygg_{yyy}-yg_yg_{yy},\quad\, &(t,y)\in\mathbb{R}^+\times\mathbb{R}.
\\
g(0,y)=g_0(y),\qquad &y\in\mathbb{R}.
\end{cases}
\end{equation}
We prove the following result.

\begin{maintheorem*}
Let $\phi_0(y) = g_0(y)-g_0''(y)$. Suppose $g_0$ is $C^2$ and $\phi_0$  satisfies the decay condition $\phi_0(y) = O(1/y^2)$ as $\lvert y\rvert \to \infty$. Then there is a $T>0$ such that there exists a unique solution of \eqref{main} on $[0,T)\times \mathbb{R}$ with $g(t)\in C^2$ for each $t$. If $\phi_0$ (and hence $g_0$) is nonnegative, then solutions exist globally. If $g_0$ is even and negative, then solutions blow up at some $T$ in the sense that $g(t,y)\to -\infty$ as $t\nearrow T$ for every $y\in \mathbb{R}$.
\end{maintheorem*}

Equation \eqref{main} is a special case of the Euler-Arnold equation on the contactomorphism group $\Diff_{\theta}(M)$:
\begin{equation}
\label{general}
m_t + u(m) + (n+2) \lambda m = 0,
\end{equation}
where $M$ is a Riemannian manifold of odd dimension $2n+1$ with a $1$-form $\theta$ satisfying $\theta\wedge (d\theta)^n \ne 0$. Here $f\colon M\to \mathbb{R}$ is a stream function, while $u=S_{\theta}f$ is a contact vector field (satisfying the condition that $\Lie_u\theta$ is proportional to $\theta$); the field $u$ is uniquely determined by $f$ via the condition $f=\theta(u)$. We denote by $\lambda$ the function such that $\Lie_u\theta = \lambda \theta$, and we write $\widetilde{S_{\theta}}f = (S_{\theta}f, \lambda)$. A Riemannian metric on $M$ determines a right-invariant Riemannian metric on $\Diff(M)\times C^{\infty}(M)$, which allows us to define $m=\widetilde{S_{\theta}}^*\widetilde{S_{\theta}}f$, which is called the contact Laplacian. See \cite{ebinpreston} for the derivation and local well-posedness theory of this equation when $M$ is compact, along with other examples. 

When $n=0$ (so that $M$ is $\mathbb{R}$ or $S^1$) and $\theta = dx$, we have $u=f \, \partial_x$, $\lambda = f_x$, and $m = f-f_{xx}$, and equation \eqref{general} becomes the Camassa-Holm (CH) equation \cite{camassaholm}
\begin{equation}\label{camassaholm}
f_t - f_{txx} + 3 ff_x - 2 f_x f_{xx} - f f_{xxx} = 0.
\end{equation}
Equation \eqref{general} can thus be considered a generalization of \eqref{camassaholm} to higher dimensions; it shares some of the same conservation laws and also has features in common with hydrodynamics. The fact that \eqref{camassaholm} is the Euler-Arnold equation of $\Diff(S^1)$ with right-invariant $H^1$ metric is due to Misio{\l}ek \cite{misiolekCH} and Kouranbaeva \cite{kouranbaevaCH}.

Equation \eqref{main} arises in the case where $n=1$ and $M=\mathbb{R}^3$ (viewed as the Heisenberg group) with the ``standard'' contact form $\theta = dz - y \,dx$. Here $u = -f_y \, \partial_x + (f_x + y\, f_z) \, \partial_y + (f-yf_y) \, \partial_z$ and $\lambda = f_z$. If the  Riemannian metric is $ds^2 = dx^2 + dy^2 + (dz-y\,dx)^2$, the natural left-invariant metric on the Heisenberg group which makes $M$ a Sasakian manifold (see Boyer~\cite{boyer}), we will have $ m = f-f_{yy} - (1+y^2)f_{zz} - 2y f_{xz} - f_{xx}$. The ansatz 
\begin{equation}
\label{form}
f(t,x,y,z)=zg(t,y),
\end{equation}
gives $m = z (g-g_{yy})$, $u = - zg_y \, \partial_x + y g\, \partial_y + z(g-yg_y)\,\partial_z$, and $\lambda = z g$, and equation \eqref{general} reduces to \eqref{main}. A similar ansatz in ideal hydrodynamics leads to the Proudman-Johnson equation, which has been studied in \cite{childress, proudmanjohnson, saxtontiglay, constantin3, sarria, cao}. The Proudman-Johnson equation was originally derived from the incompressible Euler equations by considering velocity fields of the form $\boldsymbol{u}(t,x,y)=(f(t,x),-yf_x(t,x))$, also known as stagnation-point similitude, which arise from a stream function $\psi(t,x,y)=yf(t,x)$ on an infinitely long 2D channel $(x,y)\in[0,L]\times\mathbb{R}$.

The outline of the paper is as follows. In \S\ref{sec:local} we establish some conservation laws and a local existence result for \eqref{main}; the results of \cite{ebinpreston} do not apply here since our $M$ is not compact, so we give an independent proof. In \S\ref{sec:global} we prove global existence of solutions to \eqref{main} for a class of initial data satisfying a particular sign condition. Finally in \S\ref{sec:blowup} we demonstrate the existence of solutions of \eqref{main} which blow up in finite time from smooth initial data.

\section{Local Existence}
\label{sec:local}

First we derive some preliminary results. Set
\begin{equation}
\label{phi}
\begin{split}
\phi(t,y)=g(t,y)-g_{yy}(t,y)
\end{split}
\end{equation}
and
\begin{equation}
\label{initial}
\begin{split}
\phi_0(y)=\phi(0,y)=g_0(y)-g_0''(y),\qquad\quad g_0(y)=g(0,y).
\end{split}
\end{equation}
Analogous to the Camassa-Holm equation, we may refer to \eqref{phi} as the momentum associated to the velocity $g$. In terms of the momentum, equation \eqref{main} may be rewritten in the form 
\begin{equation}
\label{phieq}
\begin{split}
\phi_t+yg\phi_y=\left(yg_y-4g\right)\phi.
\end{split}
\end{equation}
Note that we may determine $g$ from $\phi$ using the explicit solution formula 
\begin{equation}\label{phitog}
g(t,y) = \frac{1}{2} \int_{-\infty}^{\infty} e^{-\lvert y-y'\rvert} \phi(t,y') \, dy'.
\end{equation}

The characteristics are given by the solution of the flow equation 
\begin{equation}\label{gamma}
\frac{\partial \gamma}{\partial t}(t,y) = \gamma(t,y) g\big( t, \gamma(t,y)\big), \qquad \gamma(0,y) = y,
\end{equation}
and in terms of the flow $\gamma$ we obtain a convenient formula for the momentum conservation law. This formula should be considered analogous to the momentum transport law for Camassa-Holm and for the vorticity conservation law for the Euler equation of ideal hydrodynamics, in the sense that they all express the Noetherian conservation law arising from right-invariance of a Riemannian metric on a diffeomorphism group; see for example Arnold-Khesin~\cite{arnoldkhesin}.

\begin{proposition}\label{phiconservation}
If $\phi$ satisfies \eqref{phieq} and $\gamma$ satisfies \eqref{gamma}, then we have 
\begin{equation}\label{phisoln}
\phi\big(t, \gamma(t,y)\big) = \frac{\phi_0(y) y^5 \gamma_y(t,y)}{\gamma(t,y)^5}.
\end{equation}
\end{proposition}

\begin{proof}
Using the chain rule, we have 
\begin{equation}\label{lagrangianphi}
\begin{split} 
\frac{\partial}{\partial t} \log \phi\big(t, \gamma(t,y)\big) &= \frac{\phi_t\big(t, \gamma(t,y)\big) + \gamma(t,y) g\big(t, \gamma(t,y)\big) \phi_y\big(t, \gamma(t,y)\big)}{\phi\big(t, \gamma(t,y)\big)} \\
&= \gamma(t,y) g_y(t, \gamma(t,y)\big) - 4 g\big(t,\gamma(t,y)\big),
\end{split}
\end{equation}
using \eqref{phieq}.
Differentiating \eqref{gamma} with respect to $y$ we have 
\begin{equation}\label{gammaderivative}
\gamma_{ty}(t,y) = \gamma_y(t,y) g\big(t, \gamma(t,y)\big) + \gamma(t,y) \gamma_y(t,y) g_y\big(t, \gamma(t,y)\big),
\end{equation}
which we can use to eliminate both $g$ and $g_y$ in \eqref{lagrangianphi}. We obtain
$$ \frac{\partial}{\partial t} \log \phi\big(t, \gamma(t,y)\big) = \frac{\gamma_{ty}(t,y)}{\gamma_y(t,y)} - 5 \, \frac{\gamma_t(t,y)}{\gamma(t,y)},$$
which can be easily integrated to obtain \eqref{phisoln}. 
\end{proof}

A simple consequence of Proposition \ref{phiconservation} is the conservation of the sign of the momentum, which is important for our global existence results.

\begin{lemma}\label{signconservation}
Suppose \eqref{main} has a solution on $[0,T)$ for some $T>0$. Then the flow $\gamma(t,y)$ is a strictly increasing diffeomorphism of $\mathbb{R}$ with $\gamma(t,0)=0$ for all $t\in [0,T)$. Furthermore if $\phi_0(y)\ge 0$ for all $y\in \mathbb{R}$, then $\phi(t,y)\ge 0$ and $g(t,y)\ge 0$ for all $t\in [0,T)$ and $y\in \mathbb{R}$. Similarly if $\phi_0(y)\le 0$ for all $y\in \mathbb{R}$, then $\phi(t,y)$ and $g(t,y)$ are nonpositive. 
\end{lemma}

\begin{proof}
From equation \eqref{gammaderivative} we see that 
\begin{equation}\label{gammaderivativesoln}
\gamma_y(t,y) = \exp{\left( \int_0^t g\big(\tau, \gamma(\tau, y)\big) + \gamma(\tau, y) g_y\big(\tau, \gamma(\tau,y)\big) \, d\tau\right)},
\end{equation}
so that $\gamma_y(t,y)>0$ for all $t$ and $y$. Since $\gamma(0,0)=0$ we obviously have $\gamma(t,0)=0$ for all time, and we conclude that $\gamma(t,y)>0$ 
if $y>0$ and $\gamma(t,y)<0$ 
if $y<0$. Formula \eqref{phisoln} then implies that $\phi\big(t, \gamma(t,y)\big)$ has the same sign as $\phi_0(y)$ for every $y \in \mathbb{R}$. If $\phi(t,y)\ge 0$ for all $y\in\mathbb{R}$, the explicit solution formula \eqref{phitog} shows that $g(t,y)\ge 0$ as well.
\end{proof}

Another consequence of Proposition \ref{phiconservation} is the local existence theorem, which we establish by writing everything in terms of $\gamma$ as a ``particle trajectory equation'' and using Picard iteration, as in Chapter 4 of Majda-Bertozzi \cite{majdabertozzi}.

\begin{theorem}\label{localexistence}
Suppose $g_0\colon \mathbb{R}\to\mathbb{R}$ is a $C^2$ function such that $\phi_0 = g_0 - g_0''$ satisfies 
\begin{equation}\label{decaycondition}
\sup_{y\in \mathbb{R}} y^2 \lvert \phi_0(y)\rvert \le M \quad \text{for some $M$}.
\end{equation} 
Then there is a unique solution $g$ of equation \eqref{main} defined on $[0,T)\times \mathbb{R}$ for some $T>0$ such that $g(t,y)$ is $C^2$ in $y$ for each $t\in [0,T)$. 
\end{theorem}

\begin{proof}
By formula \eqref{phitog} we have 
\begin{equation}\label{convolutioncomposition}
g\big(t, \gamma(t,y)\big) = \tfrac{1}{2} \int_{-\infty}^{\infty} e^{-\lvert \gamma(t,y)-y'\rvert} \phi(t,y') \, dy' = \tfrac{1}{2} \int_{-\infty}^{\infty} e^{-\lvert \gamma(t,y)-\gamma(t,z)\rvert} \phi\big(t, \gamma(t,z)\big) \gamma_z(t,z) \, dz.
\end{equation}
Using formula \eqref{phisoln} and plugging this into \eqref{gamma}, we obtain the differential equation 
\begin{equation}
\frac{\partial \gamma}{\partial t}(t,y) = \tfrac{1}{2} \gamma(t,y) \int_{-\infty}^{\infty} e^{-\lvert \gamma(t,y)-\gamma(t,z)\rvert} \phi_0(z) \left[ \frac{z}{\gamma(t,z)}\right]^5 \gamma_z(t,z)^2 \, dz.
\end{equation}

We now view this as the equation 
\begin{equation}\label{ODE}
\frac{d\gamma}{dt} = F\big(\gamma(t)\big), \qquad \gamma(0) = y\mapsto y
\end{equation} on a certain open subset of a Banach space, where the function $F$ is given by 
\begin{equation}\label{Fdef}
F(\gamma) = y\mapsto \tfrac{1}{2} \gamma(y) \int_{-\infty}^{\infty} e^{-\lvert \gamma(y)-\gamma(z)\rvert} \phi_0(z) \left[ \frac{z}{\gamma(z)}\right]^5 \gamma'(z)^2 \, dz.
\end{equation}
Define a Banach space $B$ by 
$$B = \Big\{\gamma\colon \mathbb{R}\to \mathbb{R} \, \Big\vert \, \gamma(0)=0 \text{ and } \sup_{y\in\mathbb{R}} \,\lvert \gamma'(y)\rvert < \infty \Big\},$$
with norm $\lVert \gamma\rVert = \sup_{y\in \mathbb{R}} \lvert \gamma'(y)\rvert$. 
For numbers $a$ and $b$ satisfying $0<a<1<b$, let $U$ denote the open subset $U = \{ \gamma \in B \, \big\vert \, a<\gamma'(y)<b\,\, \forall\, y\in \mathbb{R}\}$.
Clearly $U$ contains the identity, which is the initial condition for \eqref{ODE}. Our goal is to show that $F$ is Lipschitz on $U$, and we do this by showing that $F$ has a uniformly bounded derivative on $U$. 

If $v\in B$, we easily compute that 
\begin{equation}\label{Fderivative}
[DF_{\gamma}(v)](y) = \tfrac{1}{2} \int_{-\infty}^{\infty} e^{-\lvert \gamma(y)-\gamma(z)\rvert} \zeta(z) \Big[ v(y) - 5 v(z)/\gamma(z) + 2 v'(z)/\gamma'(z)
+ \gamma(y) \sgn{(y-z)} [v(z) -v(y)]\Big] \, dz.\\
\end{equation}
where 
\begin{equation}\label{zeta}
\zeta(z) = \phi_0(z) 
[z/\gamma(z)]^5
\gamma'(z)^2.
\end{equation}
Writing $w(y) = [DF_{\gamma}(v)](y)$, we just need to show that $\lvert w'(y)\rvert$ is bounded. 
The computation of $w'$ is tedious but straightforward. Using $\lvert v'(y)\rvert \le c$ and $a\le \gamma'(y)\le b$, we obtain the estimate
$$ \lvert w'(y)\rvert \le 
\frac{b^2c}{2 a^6} \int_{-\infty}^{\infty} e^{-a\lvert y-z\rvert} \phi_0(z)\big[ a(1+3b\lvert y\rvert +b^2y^2) + (1+b\lvert y\rvert )(ab\lvert z\rvert +5b+2a)\big] \, dz.$$
Now by assumption we have $\lvert \phi_0(z)\rvert \le M'/(1+\lvert z\rvert)^2$ for some constant $M'$, and so we have 
\begin{equation}\label{wderivativeestimate}
\lvert w'(y)\rvert \le C_1 (1+\lvert y\rvert)^2 \int_{-\infty}^{\infty} \frac{e^{-a\lvert y-z\rvert}\,dz}{(1+\lvert z\rvert)^2} + C_2 (1+\lvert y\rvert) \int_{-\infty}^{\infty} \frac{e^{-a\lvert y-z\rvert}\,dz}{1+\lvert z\rvert}.
\end{equation}

The right side of \eqref{wderivativeestimate} is bounded, for we can break it up into terms that have finite limits as $\lvert y\rvert\to\infty$, as follows:
$$ \lim_{\lvert y\rvert \to\infty} \frac{\int_{-\infty}^y e^{az} \phi(z)\,dz}{e^{ay}\phi(y)} = \lim_{\lvert y\rvert \to \infty} \frac{e^{ay} \phi(y)}{e^{ay}(\phi'(y)+a\phi(y)} = \frac{1}{a}$$
by L'Hopital's rule since $\phi'(y)/\phi(y)\to 0$ for $\phi(y) = (1+\lvert y\rvert)^{-k}$ when $k=1$ or $k=2$. 
Hence $DF_{\gamma}$ is a bounded linear operator in $B$ whenever $\gamma\in U$, and thus $F$ is Lipschitz on $U$ by the Mean Value Theorem. By Picard's Theorem~\cite{hartman}, there is a unique solution for possibly short time with $\gamma(0)$ the identity.
\end{proof}

Next we establish 
the non-existence of breaking wave solutions to \eqref{main} for initial data $g_0$ such that $\phi_0\nequiv0$ does not change sign. Recall that the wave breaking phenomenon for a nonlinear wave equation is the existence of a blowup time $T$ such that $\lvert g_y(t,y_*)\rvert \to \infty$ as $t\to T$ for some $y_*$ while $\lvert g(t,y_*)\rvert$ remains bounded.

\begin{lemma}
\label{lemma2}
Let $T>0$ denote the maximal life-span of $g$. If $g_0(y)$ is such that $\phi_0(y)=g_0(y)-g_0''(y)$ does not change sign, then $\lvert g_y(t,y)\rvert \le \lvert g(t,y)\rvert$ for all $t \in [0,T)$ and $y\in \mathbb{R}$. Hence no singularity in the form of wave breaking can occur on $(0,T]$. 
\end{lemma}

\begin{proof}
Writing formula \eqref{phitog} in the form
$$ g(t,y) = \tfrac{1}{2} e^{-y} \int_{-\infty}^y e^{y'} \phi(t,y') \, dy' + \tfrac{1}{2} e^y \int_y^{\infty} e^{-y'} \phi(t,y') \, dy',$$
it is easy to check that 
\begin{equation}\label{ggprime}
g(t,y)^2 - g_y(t,y)^2 = \left(\int_{-\infty}^y e^{y'} \phi(t,y') \, dy'\right) \left( \int_y^{\infty} e^{-y'} \phi(t,y') \,dy'\right).
\end{equation}

As a result, if $\phi_0$ never changes sign, then by Lemma \ref{signconservation} we know that $\phi$ is either nonnegative or nonpositive for all $(t,y)\in [0,T)\times\mathbb{R}$. Hence the right side of \eqref{ggprime} is nonnegative and we have the bound $\lvert g_y(t,y)\rvert \le \lvert g(t,y)\rvert$ for all $y$. 
%
%
\end{proof}

For the Camassa-Holm equation \eqref{camassaholm}, it is known \cite{constantin2, mckean} that if the initial momentum does not change sign, then the solution of the equation is global in time. For equation \eqref{main}, we will see that even when the sign of the momentum is assumed constant, the behavior may be very different depending on whether it is positive or negative.

\section{Global Existence}
\label{sec:global}

In this section we study global existence of certain solutions to \eqref{main}. 
Theorem \ref{theorem1} below establishes global existence in time of solutions to \eqref{main} arising from initial data $g_0$ such that $\phi_0$ is nonnegative.

\begin{proposition}
\label{prop1}
Suppose $g$ is a solution of \eqref{main} with initial condition $\phi_0$ satisfying 
the decay condition \eqref{decaycondition}.
Then 
\begin{equation}
\label{4}
\begin{split}
\int_{\mathbb{R}}{\phi(t,y)\,dy}
\leq
\int_{\mathbb{R}}{\phi_0(y)\,dy},\qquad\quad t\in[0,T).
\end{split}
\end{equation}
\end{proposition}

\begin{proof}
We first observe that by equation \eqref{phisoln}, if $\phi_0$ satisfies the decay condition then so does $\phi(t,y)$ for any $y$. 
Using \eqref{phi} we may write \eqref{main} as
\begin{equation}
\label{1.1}
\begin{split}
\phi_t+4g\phi=y\partial_y\left(gg_{yy}-g_y^2\right).
\end{split}
\end{equation}
Integrating the right side of \eqref{1.1} using the decay condition, we obtain 
\begin{equation}
\label{3}
\begin{split}
\int_{\mathbb{R}}{y\partial_y\left(gg_{yy}-g_y^2\right)\,dy}=2\int_{\mathbb{R}}{g_y^2\,dy},
\end{split}
\end{equation}
which then yields
\begin{equation}
\label{42}
\begin{split}
\frac{d}{dt}\int_{\mathbb{R}}{g\,dy}=-2\left(2\int_{\mathbb{R}}{g^2\,dy}+\int_{\mathbb{R}}{g_y^2\,dy}\right)<0.
\end{split}
\end{equation}
The result follows using the fact that $\int_{\mathbb{R}} \phi \, dy = \int_{\mathbb{R}} g\,dy$. 
\end{proof}

Amongst other conserved quantities, the integral of the momentum $\phi$ associated to the Camassa-Holm equation is known to be conserved \cite{constantin2}. The same is true in general for equation \eqref{general} on a compact manifold, but not in this case since the equation really ``lives'' on $\mathbb{R}^3$. Essentially what is happening is that the solutions of the form \eqref{form} are not well-behaved since for example they have infinite energy; the same type of phenomenon appears when solving the Euler equation of ideal hydrodynamics: in two dimensions finite-energy solutions exist globally \cite{bardos}, but there are infinite-energy solutions of the form \eqref{form} which can blow up in finite time; see for instance \cite{childress, sarria} and references therein.


\begin{theorem}
\label{theorem1}
Suppose $\phi_0$ satisfies the decay condition \eqref{decaycondition} 
and $\phi_0(y)\geq 0$ for all $y\in\mathbb{R}$. Then the solution $g$ of \eqref{main} with initial data $g_0$ exists globally in time. 
\end{theorem}

\begin{proof}
By the general theory of ODEs in Banach spaces (see e.g. \cite{hartman}), the only way a solution of $\gamma'(t) = F(\gamma(t))$ can blow up in finite time is if it leaves all of the bounded sets on which $F$ is Lipschitz. Recall that we established the Lipschitz property under the condition that $a \le \gamma_y(t,y) \le b$ for some constants $a$ and $b$ satisfying $0<a<1<b$, so we will have global existence as long as $\gamma_y(t,y)$ does not approach either zero or infinity in finite time.

From \eqref{gammaderivativesoln} we see that a uniform bound on both $g$ and $g_y$ is sufficient to control $\gamma_y$. By Lemma \ref{signconservation} we know that $g(t,y)\ge 0$ and $\phi(t,y)\ge 0$ for all $t$ and $y$. 
%
%
Using this and the decay condition \eqref{decaycondition} to ensure that $\phi_0$ is in $L^1$, \eqref{4} implies that
\begin{equation}
\label{globalphipos}
\begin{split}
\left\|\phi(t)\right\|_{L^1}\leq\left\|\phi_0\right\|_{L^1}<\infty, \qquad\quad t\in[0,T).
\end{split}
\end{equation}
Using \eqref{phitog} we obtain
\begin{equation}
\label{globalgpos0}
\lvert g(t,y)\rvert \le \tfrac{1}{2} \int_{-\infty}^{\infty} \lvert \phi(t,y)\rvert \, dy = \tfrac{1}{2} \lVert \phi(t)\rVert_{L^1} \le \tfrac{1}{2} \lVert \phi_0\rVert_{L^1},
\end{equation}
from which we conclude 
by Lemma \ref{lemma2} that $g_y$ is uniformly bounded, and thus no blowup can occur. 
\end{proof}

\section{Blowup}
\label{sec:blowup}

Constantin-Escher \cite{constantin2, constantin1} showed that solutions of the Camassa-Holm equation \eqref{camassaholm} cannot persist globally in time if the initial data $f_0$ is odd and satisfies $f_0'(0)<0$. Further, they derived an upper bound, $T(f_0)=1/2|f_0'(0)|$, for the maximal time of existence of solutions. Theorem \ref{blowup} below establishes the existence of solutions to \eqref{main} which blow up in finite time from initial data $g_0$ that is both symmetric about $y=0$ and satisfies $g_0(0)<0$. Although singularities may form in \eqref{main} from nonpositive initial data\footnote[1]{A blowup feature that solutions to equations \eqref{main} and \eqref{camassaholm} do not share.}, we note that solutions of \eqref{main} seem to retain a few properties that are inherent to the blowup mechanism of \eqref{camassaholm}. For instance, an upper bound for the maximal time of existence of a solution to \eqref{main} is, analogous to that of \eqref{camassaholm}, given by $1/\sqrt{6}|g_0(0)|$. Here $g_0(0)<0$ serves as analogue to the Camassa-Holm condition $f_0'(0)<0$ and blowup, in both cases, is to negative infinity. Moreover, in \cite{constantin1} it was shown that if the Camassa-Holm initial profile $f_0$ is even instead of odd, with $f_0'(0)$ negative enough, then $f_x(t,0)$ can still diverge to negative infinity. A main difference between the qualitative behavior of blowup solutions to \eqref{main} and \eqref{camassaholm} is established in the second part of Theorem \ref{blowup}. More particularly, since \eqref{main} preserves the symmetry of the initial condition, for $\phi_0$ both symmetric and nonpositive we show that solutions of \eqref{main} will actually diverge everywhere on $\mathbb{R}$.
\begin{theorem}
\label{blowup}
Suppose $g$ is a solution of \eqref{main} with initial condition $\phi_0=g_0-g_0''$ satisfying the decay condition \eqref{decaycondition}. Furthermore, assume $g_0$ is even through $y=0$ and $g_0(0)<0$. Then $g(t,0)\to-\infty$ as $t\nearrow T\leq1/\left(\sqrt{6}|g_0(0)|\right)$. Additionally, if $g_0$ is such that $\phi_0\leq0$, then as $t\nearrow T$ we have $g(t,y)\to-\infty$ for all $y\in\mathbb{R}$. 

\end{theorem}

\begin{proof}
Suppose $g_0$ is even through $y=0$ and $g_0(0)<0$; let $T>0$ denote the maximal life-span of $g$. Observe that \eqref{main} may be written as
\begin{equation}
\label{eq1}
g_t(t,y)=-g(t,y)^2-yg(t,y)g_y(t,y)-p*\left[4g_y^2+3g^2+y\partial_y\left(2g_y^2-\frac{1}{2}g^2\right)\right]
\end{equation}
for $p(y)=\frac{1}{2}e^{-|y|}$. Integrating the last term in the bracket by parts, setting $y=0$, and using symmetry of $g$ about $y=0$ then yields
\begin{equation}
\label{eq2}
\begin{split}
g_t(t,0)=
-g(t,0)^2-\tfrac{1}{2}\int_{0}^{\infty}{\left(4g_s^2+7g^2+4sg_s^2-sg^2\right)e^{-s}ds}
\end{split}
\end{equation}
for all $t\in(0,T)$. Now set 
\begin{equation}
\label{h}
\begin{split}
g(t,y)=e^{y/2}h(t,y)
\end{split}
\end{equation}
and note that vanishing of $ye^{-y}g(t,y)^2$ as $y\to\infty$ implies the same for $yh(t,y)^2$. Using \eqref{h} on \eqref{eq2} we obtain
\begin{equation}
\label{integrand}
\begin{split}
g_t(t,0)=
-g(t,0)^2-2\int_{0}^{\infty}{\left(2h^2+h_s^2+sh_s^2+ hh_s+shh_s\right)ds}.
\end{split}
\end{equation}
Integrating the last two terms in \eqref{integrand} by parts and using the above decay condition of $h$ now implies
\begin{equation}
\label{integrand2}
\begin{split}
g_t(t,0) &= -2\int_0^{\infty} h_s^2 \, ds - 3 \int_0^{\infty} h^2 \,ds - 2\int_0^{\infty} sh_s^2 \, ds \\
&\le 2\sqrt{6} \int_0^{\infty} hh_s \, ds - 2\int_0^{\infty} sh_s^2 \,ds \\
&\le -\sqrt{6} g(t,0)^2,
\end{split}
\end{equation}
from which we conclude that 
$g(t,0)\to-\infty$ as $t\nearrow T\leq 1/\big(\sqrt{6} |g_0(0)|\big)$.

Now suppose $\phi_0(y)\leq0$. Then Lemma \ref{signconservation} and the bound $\lvert g_y(t,y)\rvert \le \lvert g(t,y)\rvert$, which was established in the proof of  Lemma \ref{lemma2}, imply that 
\begin{equation}
\label{last}
\begin{split}
\lvert g(t,0)\rvert e^{-y}\leq \lvert g(t,y)\rvert \leq \lvert g(t,0)\rvert e^{y},\qquad y\geq0,\quad t\in[0,T).
\end{split}
\end{equation}
Letting $t\nearrow T$ in \eqref{last} we see that $g(t,y)\to-\infty$ for all $y\geq0$. Symmetry of $g$ about $y=0$ then yields our result for all $y\in\mathbb{R}$.

\end{proof}

\end{document}